\documentclass{amsart}
\usepackage{mathptmx, amssymb, enumerate, colonequals, xcolor}
\usepackage{color}
\definecolor{chianti}{rgb}{0.6,0,0}
\definecolor{meretale}{rgb}{0,0,.6}
\definecolor{leaf}{rgb}{0,.35,0}
\usepackage[colorlinks=true, pagebackref, hyperindex, citecolor=meretale, urlcolor=leaf, linkcolor=chianti]{hyperref}

\newtheorem{theorem}{Theorem}[section]
\newtheorem{lemma}[theorem]{Lemma}

\theoremstyle{definition}

\newtheorem{remark}[theorem]{Remark}
\newtheorem{conjecture}[theorem]{Conjecture}
\numberwithin{equation}{theorem}

\def\id{\operatorname{id}}
\def\tr{{\operatorname{tr}}}

\def\GL{\operatorname{GL}}

\def\O{\operatorname{O}}

\def\SO{\operatorname{SO}}

\def\ge{\geqslant}
\def\le{\leqslant}
\def\tilde{\widetilde}
\def\to{\longrightarrow}

\def\mapsto{\longmapsto}

\def\Hom{\operatorname{Hom}}
\def\HH{\underline{\mathrm{Hom}}}

\def\fraka{\mathfrak{a}}
\def\frakm{\mathfrak{m}}
\def\frakp{\mathfrak{p}}

\def\AA{\mathbb{A}}

\def\NN{\mathbb{N}}
\def\QQ{\mathbb{Q}}
\def\ZZ{\mathbb{Z}}

\begin{document}
\title[{Invariant rings of the special orthogonal group have nonunimodal $h$-vectors}]{Invariant rings of the special orthogonal group\\ have nonunimodal $h$-vectors}

\author{Aldo Conca}
\address{Dipartimento di Matematica, Universit\`a di Genova, Dipartimento di Eccellenza 2023-2027, Via Dodecaneso 35, I-16146 Genova, Italy}
\email{conca@dima.unige.it}

\author{Anurag K. Singh}
\address{Department of Mathematics, University of Utah, 155 South 1400 East, Salt Lake City, UT~84112, USA}
\email{singh@math.utah.edu}

\author{Matteo Varbaro}
\address{Dipartimento di Matematica, Universit\`a di Genova, Dipartimento di Eccellenza 2023-2027, Via Dodecaneso 35, I-16146 Genova, Italy}
\email{varbaro@dima.unige.it}

\thanks{A.C. and M.V. are supported by PRIN~2020355B8Y ``Squarefree Gr\"obner degenerations, special varieties and related topics,'' by MIUR Excellence Department Project awarded to the Dept.~of Mathematics, Univ.~of Genova, CUP D33C23001110001 and by INdAM-GNSAGA; A.K.S. is supported by NSF grants DMS-2101671 and DMS-2349623; all authors were supported by NSF grant DMS-1928930 and by Alfred P. Sloan Foundation grant G-2021-16778, while in residence at SLMath/MSRI, Berkeley, during the Spring 2024 Commutative Algebra program. We are grateful to Mitsuyasu Hashimoto for valuable discussions, and to the referee for several useful~comments.}

\dedicatory{To Sudhir Ghorpade, in celebration of his sixtieth birthday.}

\begin{abstract}
For $K$ an infinite field of characteristic other than two, consider the action of the special orthogonal group~$\operatorname{SO}_t(K)$ on a polynomial ring via copies of the regular representation. When $K$ has characteristic zero, Boutot's theorem implies that the invariant ring has rational singularities; when $K$ has positive characteristic, the invariant ring is~$F$-regular, as proven by Hashimoto using good filtrations. We give a new proof of this, viewing the invariant ring for $\operatorname{SO}_t(K)$ as a cyclic cover of the invariant ring for the corresponding orthogonal group; this point of view has a number of useful consequences, for example it readily yields the~$a$-invariant and information on the Hilbert series. Indeed, we use this to show that the $h$-vector of the invariant ring for $\operatorname{SO}_t(K)$ need not be unimodal.
\end{abstract}
\maketitle

\section{Introduction}

Let $X$ be an $n\times n$ symmetric matrix of indeterminates over a field $K$, and let $I_{t+1}(X)$ denote the ideal of the polynomial ring $K[X]$ generated by the size $t+1$ minors of~$X$. For~$t$ a positive integer with $t+1\le n$, we refer to $K[X]/I_{t+1}(X)$ as a \emph{symmetric determinantal ring}. The ring~$K[X]/I_{t+1}(X)$ is a Cohen-Macaulay normal domain of dimension
\[
\binom{n+1}{2} - \binom{n+1-t}{2},
\]
as proven in~\cite{Kutz}. These rings have been studied extensively, in part because they arise as invariant rings for the natural action of the orthogonal group
\begin{equation}
\label{equation:orthogonal}
\O_t(K)\colonequals\{M\in\GL_t(K)\ |\ M^\tr M=\id\}
\end{equation}
as follows: for $Y$ a $t\times n$ matrix of indeterminates, $\O_t(K)$ acts~$K$-linearly on $K[Y]$ via
\[
M\colon Y\mapsto MY\qquad\text{ for }\ M\in\O_t(K).
\]
This is a right action of $\O_t(K)$ on the polynomial ring $K[Y]$, corresponding to a left action of $\O_t(K)$ on the affine space $\AA_K^{t\times n}$. Note that $Y^\tr Y\mapsto Y^\tr M^\tr MY=Y^\tr Y$ for $M\in\O_t(K)$, so the entries of $Y^\tr Y$ are invariant under the action; when the field $K$ is infinite of characteristic other than two, the invariant ring is precisely the $K$-algebra generated by the entries of $Y^\tr Y$, see~\cite[Theorem~5.6]{DeConcini-Procesi}, and is isomorphic to the symmetric determinantal ring $K[X]/I_{t+1}(X)$ via the entrywise map $X\mapsto Y^\tr Y$. We use this to identify the rings~$K[X]/I_{t+1}(X)$ and $K[Y^\tr Y]$.

By \cite{Goto1, Goto2}, the ring~$R\colonequals K[Y^\tr Y]$ has class group $\ZZ/2$, and is Gorenstein precisely when $n\equiv t+1\mod 2$. Taking $\frakp$ to be a prime ideal that serves as a generator for the class group, it follows that the symbolic power $\frakp^{(2)}$ is isomorphic to $R$. We choose an explicit isomorphism $\frakp^{(2)}\cong R$ so that the cyclic cover of $R$ with respect to $\frakp$ is precisely the invariant ring for the action of the special orthogonal group $\SO_t(K)$. This gives a straightforward approach towards studying the invariant ring~$K[Y]^{\SO_t(K)}$, for example towards determining its $a$-invariant and information regarding the Hilbert series.

When $K$ is an infinite field of characteristic two, the groups $\O_t(K)$ and $\SO_t(K)$ coincide when taking $\O_t(K)$ to be the group as defined in~\eqref{equation:orthogonal}; the invariant ring in this case is
\[
K[Y^\tr Y,\ \sum_{i=1}^t y_{ij}\ |\ 1\le j\le n],
\]
see~\cite[Proposition~17]{Richman}, and a presentation is provided by~\cite[Proposition~23]{Richman}. The reader is warned that there are varying definitions used for the orthogonal group in characteristic two, see for example~\cite[page~10]{ag:4}.

Section~\ref{section:cover} includes some generalities on cyclic covers; these are used in Section~\ref{section:a:invariant} where we compute the $a$-invariant of $K[Y]^{\SO_t(K)}$ and also record a proof that this ring is $F$-regular. Section~\ref{section:h:vector} is devoted to the $h$-vector of $K[Y]^{\SO_t(K)}$, i.e., the coefficients of the numerator of its Hilbert series: the key result here is that this invariant ring is a semistandard graded Gorenstein normal domain, for which the $h$-vector need not be unimodal; the context for this is discussed as well in Section~\ref{section:h:vector}.

\section{Cyclic covers and \texorpdfstring{$F$}{F}-regularity}
\label{section:cover}

Let $R$ be a normal domain. By a \emph{divisorial ideal} of $R$, we mean a nonzero intersection of fractional principal ideals. Let $\fraka$ be a divisorial ideal that has finite order $m$ when viewed as an element of the divisor class group of $R$. Then $\fraka^{(m)}=\alpha R$, for an element $\alpha$ in the fraction field of $R$. Set
\begin{equation}
\label{equation:t:root}
T\colonequals 1/\alpha^{1/m},
\end{equation}
which is an element in an algebraic closure of the fraction field of $R$; the choice of $\alpha$ or the~$m$-th root is not unique. The \emph{cyclic cover} of $R$ with respect to $\fraka$ is the ring 
\[
\tilde{R}\colonequals R[\fraka T, \ \fraka^{(2)}T^2, \ \fraka^{(3)}T^3,\ \dots],
\]
viewed as a subring of $R[T]$. Since
\[
\fraka^{(m+k)}T^{m+k}\ =\ \alpha \fraka^{(k)}T^{m+k}\ =\ \fraka^{(k)}T^k
\]
for each $k\ge 0$, the ring $\tilde{R}$ is a finitely generated reflexive $R$-module; specifically, one has an $R$-module isomorphism
\[
\tilde{R}\ \cong\ R\oplus \fraka\oplus \fraka^{(2)}\oplus\cdots\oplus \fraka^{(m-1)}.
\]

When the ring $R$ is $\NN$-graded and $\fraka$ is a homogeneous divisorial ideal of finite order~$m$, there exists a homogeneous element $\alpha$ with $\fraka^{(m)}=\alpha R$, and the $\NN$-grading on $R$ extends to a~$\QQ$-grading on $\tilde{R}$ obtained by setting
\[
\deg T \colonequals -(\deg\alpha)/m.
\]
It turns out that this is a $\QQ_{\ge0}$-grading on $\tilde{R}$, and that ${[\tilde{R}]}_0=R_0$, see~\cite[Proposition~4.2]{Singh:cyclic}.

Suppose that the characteristic of $R$ is zero or relatively prime to $m$, and that $\frakp$ is a height one prime ideal of $R$. Then the ideal $\fraka R_\frakp$ is principal; take $r$ to be a generator. Since~$r^m=\alpha u$, for $u$ a unit in $R_\frakp$, it follows that
\[
\tilde{R}_\frakp\ =\ R_\frakp[rT]\ \cong\ R_\frakp[u^{1/m}],
\]
so $R_\frakp\to\tilde{R}_\frakp$ is \'etale. In particular, under this assumption on the characteristic, the ring~$\tilde{R}_\frakp$ is regular for each height one prime of $R$; since each $\fraka^{(k)}$ is reflexive, the ring~$\tilde{R}$ also satisfies the Serre condition $S_2$, and is hence a normal domain. By~\cite[Theorem~2.7]{watanabe:dim2}, $F$-regularity is preserved under finite extensions that are \'etale at height one primes, so one has:

\begin{theorem}[Watanabe]
\label{theorem:f:regular}
Let $R$ be an $\NN$-graded ring that is finitely generated over a field~$R_0$ of characteristic $p>0$, and let $\tilde{R}$ be the cyclic cover of $R$ with respect to a homogeneous ideal of finite order relatively prime to $p$. Then, if $R$ is $F$-regular, so is $\tilde{R}$.
\end{theorem}

The restriction on the characteristic is removed in~\cite[Theorem~C]{Carvajal}. For the theory of~$F$-regularity in the graded setting, we point the reader towards~\cite{HH:JAG}. When~$R$ is an~$\NN$-graded ring finitely generated over a field $R_0$ of positive characteristic, the notions of weak~$F$-regularity, $F$-regularity, and strong~$F$-regularity all coincide as proven in~\cite{Lyubeznik:Smith}, so we do not make a distinction between these in the present paper.

The $F$-regularity of generic determinantal rings and of Pl\"ucker coordinate rings of Grassmannians is proven as~\cite[Theorem~7.14]{HH:JAG}; the proof therein is readily adapted to symmetric determinantal rings, as we show next. For a different approach, see~\cite[\S4.1]{Lorincz}.

\begin{theorem}
\label{theorem:symmetric}
Let $X$ be an $n\times n$ symmetric matrix of indeterminates over a field $K$ of positive prime characteristic. Then the ring~$K[X]/I_{t+1}(X)$ is $F$-regular.
\end{theorem}

\begin{proof}
If $n\equiv t+1\mod 2$, then $K[X]/I_{t+1}(X)$ is Gorenstein; otherwise, enlarge $X$ to a symmetric matrix $\tilde{X}$ of size $n+1$, in which case the ring $K[\tilde{X}]/I_{t+1}(\tilde{X})$ is Gorenstein, and contains~$K[X]/I_{t+1}(X)$ as a pure subring. Since $F$-regularity is inherited by pure subrings, it suffices to prove the desired result when $R\colonequals K[X]/I_{t+1}(X)$ is Gorenstein.

The $a$-invariant of $R$ is computed in~\cite{Barile:a:inv} and~\cite{Conca:symmetric:ladders}, and recorded in the following section; in particular, $a(R)<0$. We next claim that $R$ is $F$-injective, equivalently $F$-pure, since the notions coincide in the Gorenstein case. This follows by~\cite[Theorem~2.1]{Conca:Herzog:ladder} in combination with the main result of \cite{Conca:GBsym} asserting that the ``diagonal'' initial ideal of~$I_{t+1}(X)$ is square-free and defines a Cohen-Macaulay ring. 

The $F$-regularity of $R$ now follows from~\cite[Corollary~7.13]{HH:JAG}, once we verify that the localization $R_{x_{ij}}$ is $F$-regular for each $x_{ij}$. Using the lemma below and induction on $t$, the localizations $R_{x_{11}}$ and $R_\Delta$ are $F$-regular; but then $R_\frakp$ is $F$-regular if $\frakp$ is a prime ideal such that $x_{11}\notin\frakp$ or $\Delta\notin\frakp$. It follows that $R_\frakp$ is also $F$-regular if $x_{12}\notin\frakp$. Since we have accounted for the diagonal variable $x_{11}$ and the off-diagonal variable $x_{12}$, the symmetry implies that $R_{x_{ij}}$ is $F$-regular for each $x_{ij}$.
\end{proof}

\begin{lemma}
\label{lemma:sym:invert}
Let $R\colonequals K[X]/I_{t+1}(X)$, where $X$ is a symmetric $n\times n$ matrix of indeterminates. Then:
\begin{enumerate}[\ \rm(1)]
\item The ring $R_{x_{11}}$ is isomorphic to a localization of a polynomial ring over $K[X']/I_t(X')$, where $X'$ is a symmetric $(n-1)\times(n-1)$ matrix of indeterminates.
\item For $\Delta\colonequals x_{11}x_{22}-x_{12}^2$, the ring $R_\Delta$ is isomorphic to a localization of a polynomial ring over $K[X']/I_{t-1}(X')$, for $X'$ a symmetric $(n-2)\times(n-2)$ matrix of indeterminates.
\end{enumerate}
\end{lemma}

For a proof, see~\cite[Lemma~1.1]{Jozefiak}; the argument also appears implicitly in \cite{Micali:Villamayor}.

\section{The \texorpdfstring{$a$}{a}-invariant}
\label{section:a:invariant}

Let $Y$ be a $t\times n$ matrix of indeterminates over a field $K$. In this section, we work with the grading on the subring $R\colonequals K[Y^\tr Y]$ that is induced by the standard grading on the polynomial ring $K[Y]$. Note that under the identification of $K[X]/I_{t+1}(X)$ with $K[Y^\tr Y]$, this corresponds to taking $\deg x_{ij}=2$ for each $i,j$. With this grading, \cite[Theorem~4.4]{Barile:a:inv} or~\cite[Theorem~2.4]{Conca:symmetric:ladders} imply that the $a$-invariant of $R$ is
\[
a(R)\ =\ 
\begin{cases}
-t(n+1) & \text{ if } n\equiv t\mod 2,\\
-tn & \text{ if } n\not\equiv t\mod 2;
\end{cases}
\]
more generally, the graded canonical module of $R$ is
\[
\omega_R\ =\ 
\begin{cases}
\frakp(-tn+t) & \text{ if } n\equiv t\mod 2,\\
R(-tn) & \text{ if } n\not\equiv t\mod 2,
\end{cases}
\]
where $\frakp$ is the ideal of $K[Y^\tr Y]$ generated by the maximal minors of the first $t$ rows of $Y^\tr Y$, i.e., by the maximal minors of the product matrix
\[
\begin{pmatrix}
y_{11} & y_{21} & \cdots & y_{t1}\\
y_{12} & y_{22} & \cdots & y_{t2}\\
\vdots & \vdots & \vdots & \vdots\\
y_{1t} & y_{2t} & \cdots & y_{tt}
\end{pmatrix}
\begin{pmatrix}
y_{11} & y_{12} & y_{13} & \cdots & \cdots & y_{1n}\\
y_{21} & y_{22} & y_{23} & \cdots & \cdots & y_{2n}\\
\vdots & \vdots & \vdots & \vdots & \vdots & \vdots\\
y_{t1} & y_{t2} & y_{t3} & \cdots & \cdots & y_{tn}
\end{pmatrix}.
\]
Using the identification of $K[X]/I_{t+1}(X)$ with $K[Y^\tr Y]$, the ideal $\frakp$ is prime of height one by~\cite[Theorem~1]{Kutz}, and generates the class group of~$R$ by~\cite{Goto1}. The symbolic power~$\frakp^{(2)}$ is the principal ideal of $R$ generated by the determinant of the first $t$ columns of the product matrix displayed above, i.e., $\frakp^{(2)}$ is generated by the square of
\[
\Delta\colonequals\det\begin{pmatrix}
y_{11} & y_{21} & \cdots & y_{t1}\\
y_{12} & y_{22} & \cdots & y_{t2}\\
\vdots & \vdots & \vdots & \vdots\\
y_{1t} & y_{2t} & \cdots & y_{tt}
\end{pmatrix}.
\]
Choosing a unit as in~\eqref{equation:t:root}, set
\[
T\colonequals 1/\Delta.
\]
The generators of $\frakp T$ are then identified with the maximal minors of the matrix $Y$, so that the cyclic cover $\tilde{R}$ of $R$ with respect to $\frakp$ is the subring of the polynomial ring $K[Y]$ generated by the entries of the product matrix $Y^\tr Y$ along with the maximal minors of $Y$. It is clear that these generators are fixed under the action of the special orthogonal group
\[
M\colon Y\mapsto MY\qquad\text{ for }\ M\in\SO_t(K).
\]
When the field $K$ is infinite of characteristic other than two, the invariant ring is precisely the $K$-algebra generated by these elements,~\cite[Theorem~5.6]{DeConcini-Procesi}.

We determine the graded canonical module of $\tilde{R}$; while the semisimplicity of $\SO_t(K)$ may be used to verify that $\tilde{R}$ is Gorenstein, \cite[page~123]{Hochster:Roberts}, our goal is to additionally obtain the~$a$-invariant of $\tilde{R}$. Since $\deg T=-t$, one has
\[
\tilde{R}\ =\ R\oplus\frakp(t).
\]
Let $\frakm$ denote the homogeneous maximal ideal of $R$. For an $\NN$-graded $R$-module $M$, we use $\HH(M,R/\frakm)$ to denote its graded dual as in \cite[page~184]{Goto:Watanabe}. Setting $d\colonequals\dim R$, the graded canonical module of $\tilde{R}$ may be computed as
\[
\omega_{\tilde{R}}\ =\ \HH\big(H_\frakm^d(\tilde{R}),\ R/\frakm\big)
\ =\ \HH\big(H_\frakm^d(R),\ R/\frakm\big) \oplus \HH\big(H_\frakm^d(\frakp(t)),\ R/\frakm\big).
\]
The first term in this direct sum is $\omega_R$, while the second is
\begin{align*}
\HH\big(H_\frakm^d(\frakp(t)),\ R/\frakm\big)\
& =\ \HH\big(H_\frakm^d(\omega_R)\otimes_R \omega_R^{(-1)}\otimes_R\frakp(t),\ R/\frakm\big)\\
& =\ \Hom_R\big(\omega_R^{(-1)}\otimes_R\frakp(t),\ \HH\big(H_\frakm^d(\omega_R),\ R/\frakm\big)\big)\\
& =\ \Hom_R\big(\omega_R^{(-1)}\otimes_R\frakp(t),\ R\big)\\
& =\ \big(\omega_R\otimes_R\frakp^{(-1)}(-t)\big)^{**},
\end{align*}
where $(-)^{**}$ is the reflexive hull. Since $\frakp^{(2)}=R(-2t)$, one has $\frakp^{(-1)}=\frakp(2t)$, so
\[
\big(\omega_R\otimes_R\frakp^{(-1)}(-t)\big)^{**}\ =\ 
\begin{cases}
R(-tn) & \text{ if } n\equiv t\mod 2,\\
\frakp(-tn+t) & \text{ if } n\not\equiv t\mod 2.
\end{cases}
\]
Putting it all together, one gets
\[
\omega_{\tilde{R}}\ =\ \begin{cases}
\frakp(-tn+t)\oplus R(-tn) & \text{ if } n\equiv t\mod 2,\\
R(-tn)\oplus\frakp(-tn+t) & \text{ if } n\not\equiv t\mod 2,
\end{cases}
\]
so that
\[
\omega_{\tilde{R}}\ =\ \tilde{R}(-tn),
\]
i.e., $\tilde{R}$ is Gorenstein with $a(\tilde{R})=-tn$. To summarize what we have at this stage:

\begin{theorem}
\label{theorem:a:inv}
Let $Y$ be a $t\times n$ matrix of indeterminates over a field $K$ of characteristic other than two. Let~$\tilde{R}$ denote the $K$-subalgebra of $K[Y]$ generated by the entries of the product matrix $Y^\tr Y$ along with the maximal minors of $Y$. Then $\tilde{R}$ is a Gorenstein normal domain. When $K$ has characteristic zero, the ring $\tilde{R}$ has rational singularities; when $K$ has positive characteristic, $\tilde{R}$ is $F$-regular.

With the $\NN$-grading on $\tilde{R}$ inherited from the standard grading on $K[Y]$, one has
\[
a(\tilde{R})\ =\ -tn.
\]
\end{theorem}

The fact that $\tilde{R}$ has rational singularities in characteristic zero follows from Boutot's theorem~\cite{Boutot}; the $F$-regularity in characteristic $p\ge 3$ follows by combining Theorem~\ref{theorem:f:regular} and Theorem~\ref{theorem:symmetric}. For a different approach using good filtrations, see~\cite[Corollary~2]{Hashimoto:Zeit}.

\begin{remark}
The ring $\tilde{R}$ in Theorem~\ref{theorem:a:inv} has $K$-algebra generators in degree $2$ and degree~$t$; it admits a standard grading in the following two cases:

(i) When $t=1$, index the entries of $Y$ as $y_1,\dots,y_n$. The ring $R\colonequals K[Y^\tr Y]$ is then the second Veronese subring of the polynomial ring $K[Y]$, i.e., the subring generated by the monomials~$y_iy_j$. One has
\[
\frakp\ =\ (y_1^2, y_1y_2,\dots,y_1y_n)R \qquad\text{ and }\qquad \frakp^{(2)}\ =\ (y_1^2)R.
\]
Taking $T\colonequals 1/y_1$, the cyclic cover $\tilde{R}$ coincides with $K[Y]$ under the standard grading.

(ii) When $t=2$, the $K$-algebra generators of $\tilde{R}$ are the entries of $Y^\tr Y$, and the size two minors of $Y$; these generators all have degree two, so the grading on $\tilde{R}$ may be rescaled to a standard grading.
\end{remark}

\begin{remark}
\label{remark:regrade}
When $t$ is even, the ring $\tilde{R}$ in Theorem~\ref{theorem:a:inv} has generators of even degree; rescaling by a factor of two, one obtains generators in degree one (the entries of $Y^\tr Y$) and generators in degree $t/2$ (the maximal minors of $Y$); this is the grading considered in the following section. This is a \emph{semistandard} grading on $\tilde{R}$, i.e., an $\NN$-grading under which the ring is integral over the~$K$-subalgebra generated by its elements of degree one.
\end{remark}

\section{Nonunimodal \texorpdfstring{$h$}{h}-vectors}
\label{section:h:vector}

A description for the Hilbert function of a generic determinantal ring may be found in~\cite{Abhyankar}, while an expression for its Hilbert series is presented in~\cite{Conca:Herzog:hilbert}. In particular, for the numerator of the Hilbert series, known as the \emph{$h$-polynomial}, one has both a combinatorial description (in terms on non-intersection paths with given number of turns) and an explicit compact (and determinantal!) formula. For pfaffian rings, the corresponding results are in~\cite{DeNegri, Ghorpade:Krattenthaler}. For symmetric determinantal rings one finds in \cite{Conca:symmetric:ladders} a combinatorial description of the $h$-polynomial, but no compact determinantal expression for it is known in general. However, for $X$ a symmetric $n\times n$ matrix of indeterminates and $t+1=n-1$, the expression of the $h$-polynomial of $K[X]/I_{t+1}(X)$ is easily obtained to be
\begin{equation}
\label{equation:codim:3}
\binom{2}{2}+\binom{3}{2}z+\dots+\binom{n}{2}z^{n-2},
\end{equation}
see for example \cite[Example~2.3(c)]{Conca:symmetric:ladders}. 

As in Remark~\ref{remark:regrade}, an $\NN$-grading on a ring $A$ is \emph{semistandard} if $A$ is a finitely generated algebra over a field~$K\colonequals A_0$, and $A$ is integral over the~$K$-subalgebra generated by its elements of degree one. This condition ensures that the Hilbert series of $A$ may be written as a rational function
\[
\frac{h_0+h_1z+h_2z^2+\dots+h_kz^k}{(1-z)^{\dim A}},
\qquad\text{ where }h_i \in\ZZ\text{ and }h_k\neq 0.
\]
The coefficients of the numerator, i.e., of the $h$-polynomial, form the \emph{$h$-vector} $(h_0,\dots,h_k)$ of the ring $A$. When $A$ is Cohen-Macaulay, it is readily seen that each $h_i$ is nonnegative; when~$A$ is Gorenstein, the~$h$-vector is a palindrome, i.e., $h_i=h_{k-i}$ for each $0\le i\le k$. In this case, the $h$-vector is said to be \emph{unimodal} if
\[
h_0\le h_1\le\dots\le h_{\lfloor{k/2}\rfloor}.
\]
Unimodality results reflect interesting geometric and combinatorial properties; they figure prominently in Ehrhart theory. Following his proof of the Anand-Dumir-Gupta conjectures regarding the enumeration of magic squares~\cite{Stanley73,Stanley:book}, Stanley asked if the $h$-vector of the corresponding affine semigroup ring is unimodal. This was indeed proven to be the case by Athanasiadis \cite{Athanasiadis}, see also~\cite{Bruns:Roemer}. While Musta\c t\u a and Payne \cite{Mustata:Payne} have constructed examples of Gorenstein normal affine semigroup rings for which the~$h$-vector is not unimodal, these are not standard graded, and the following remains unresolved:

\begin{conjecture}
\label{conjecture:stanley}
The $h$-vector of a standard graded Gorenstein domain is unimodal.
\end{conjecture}

This is due to Stanley~\cite[Conjecture~4(a)]{Stanley:jinan}, see also \cite[Conjecture~1]{Braun}, \cite[Conjecture~5.1]{Brenti}, \cite[page~36]{Bruns:allahabad}, and \cite[Conjecture~1.5]{Hibi}. We show that invariant rings for the action of $\SO_t(K)$ yield examples of ``naturally occurring'' semistandard graded Gorenstein normal domains, for which the $h$-vector is not unimodal:

\begin{theorem}
Consider a $2m\times (2m+2)$ matrix of indeterminates~$Y$ over a field $K$ of characteristic other than two. Let~$\tilde{R}$ denote the $K$-subalgebra of $K[Y]$ generated by the entries of the product matrix $Y^\tr Y$ and the maximal minors of $Y$, where the generators are assigned degree $1$ and degree $m$ respectively. If~$m\ge 2$, the $h$-vector of $\tilde{R}$ is not unimodal.
\end{theorem}

\begin{proof}
Viewing the subring $R\colonequals K[Y^\tr Y]$ as a symmetric determinantal ring and using the expression~\eqref{equation:codim:3}, one see that $R$ has Hilbert series
\[
\frac{\binom{2}{2}+\binom{3}{2}z+\dots+\binom{2m+2}{2}z^{2m}}{(1-z)^{2m^2+5m}}.
\]
The ring $R$ is not Gorenstein; the Hilbert series of $R$ yields that of $\omega_R$, from which it follows that the cyclic cover $\tilde{R}$ has Hilbert series
\[
\frac{\left[\binom{2}{2}+\binom{3}{2}z+\dots+\binom{2m+2}{2}z^{2m}\right] + \left[\binom{2m+2}{2}z^m+\binom{2m+1}{2}z^{m+1}+\dots+\binom{2}{2}z^{3m}\right]}{(1-z)^{2m^2+5m}}.
\]
Hence
\[
h_m-h_{m+1}\ =\ \left[\binom{m+2}{2}+\binom{2m+2}{2}\right]-\left[\binom{m+3}{2}+\binom{2m+1}{2}\right]\ =\ m-1,
\]
so the $h$-vector of $\tilde{R}$ is not unimodal; for a specific example, the case $m=2$ yields the nonunimodal~$h$-vector
\[
(1,\ 3,\ 6,\ 10,\ 15,\ 0,\ 0) + (0, \ 0, \ 15,\ 10,\ 6,\ 3,\ 1)\ =\ (1,\ 3,\ 21,\ 20,\ 21,\ 3,\ 1).
\qedhere\]
\end{proof}


\end{document}